\documentclass{amsart}
\usepackage{amssymb}

\newtheorem{thm}{Theorem}
\newtheorem*{lemma}{Lemma}
\theoremstyle{definition}
\newtheorem{rem}{Remark}
\newtheorem*{ack}{Acknowledgements}

\newcommand{\ot}{\otimes}

\begin{document}

\title{Homotopy DG algebras induce homotopy BV algebras} \author{John Terilla} \address{John Terilla, Queens College, City University of New York, Flushing NY 11367}
\email{jterilla@qc.cuny.edu} \author{Thomas Tradler} \address{Thomas Tradler,
 NYC College of Technology, City University of New York, Brooklyn NY
  11201} \email{ttradler@citytech.cuny.edu} \author{Scott O. Wilson}
\address{Scott O. Wilson, Queens College, City University of New York, Flushing NY
  11367} \email{scott.wilson@qc.cuny.edu}

\begin{abstract}
  Let $TA$ denote the space underlying the tensor algebra of a
  vector space $A$.
  In this short note, we show that if $A$ is a differential graded
  algebra, then $TA$ is a 
  differential Batalin-Vilkovisky algebra. Moreover, if $A$
  is an $A_\infty$ algebra, then $TA$ is a commutative $BV_\infty$ algebra. 
\end{abstract}

\maketitle

\section{Main Statement}

Let $(A,d_A)$ be a complex over
a commutative ring $R$.  Our convention is that 
$d_A$ is of degree $+1$. The space 
$TA=\bigoplus_{n\geq 0} A^{\otimes n}$ is graded by
declaring monomials of homogeneous elements 
$a_1 \ot \cdots \ot a_n \in A^{\otimes
  n}$ to be of degree $|a_1 | + \cdots + |a_n| + n$.

There is a shuffle product $\bullet:TA\otimes TA\to TA$ generated
by
\[
(a_1\otimes\dots\otimes a_n)\bullet(a_{n+1}\otimes\dots\otimes
a_{n+m}) := \sum_{\sigma\in S(n,m)} (-1)^\kappa \cdot
a_{\sigma^{-1}(1)}\otimes \dots\otimes a_{\sigma^{-1}(n+m)},
\]
where $S(n,m)$ is the set of all $(n,m)$-shuffles, {\it i.e.} $S(n,m)$
is the set of all permutations $\sigma\in \Sigma_{n+m}$ with
$\sigma(1)<\dots<\sigma(n)$ and $\sigma(n+1)<\dots<\sigma(n+m)$, ({\it
  cf.} \cite{M}). Here $(-1)^\kappa$ is the Koszul sign, which
introduces a factor of $(|a_i| + 1)(|a_j| +1)$ whenever the elements
$a_i$ and $a_j$ move past one another in a shuffle. Note that for
degree zero elements of $A$, this Koszul sign is just $sgn(\sigma)$,
the sign of the permutation $\sigma$.  The shuffle product makes $TA$
into a graded commutative associative algebra. Recall that $TA$ is also a coalgebra
under the usual tensor coproduct.

There is a differential $d:TA\to TA$ (of degree
$+1$) given by extending the differential $d_A:A\to A$ as a
coderivation of the tensor coproduct, see {\it e.g.} \cite{S}:
\[
d( a_1 \ot \cdots \ot a_n) = \sum_{i=0}^n (-1)^{|a_1| + \dots +
  |a_{i-1}| + i-1} a_1 \ot \cdots \ot d_A (a_i) \ot \cdots \ot a_n
\]
and together with the shuffle product, 
the triple $(TA,d,\bullet)$ is a differential graded commutative
associative algebra.

If $\mu_A:A\otimes A\to A$ is an associative product, then 
there is another differential $\Delta=\tilde \mu_A:TA\to TA$, of
degree $-1$, given by extending the multiplication as a coderivation,
\[
\Delta (a_1\otimes \dots \otimes a_n) =\sum_{i=1}^{n-1} (-1)^{|a_1| +
  \dots + |a_{i}| + i-1} a_1\otimes\dots\otimes
\mu_A(a_i,a_{i+1})\otimes \dots\otimes a_n.
\]

In Section \ref{proof1} we show:
\begin{thm}\label{theorem1}
  If $(A,d_A,\mu_A)$ is a differential graded algebra, then 
  $(TA,d,\Delta,\bullet)$ defines a dBV algebra. The
  construction is functorial: If $f: A \to B$ is a morphism of
  differential associative algebras, then the induced map from $TA$ to
  $TB$ is a morphism of dBV algebras.
\end{thm}

Recall that a dBV algebra $(X,d,\Delta,\bullet)$ is a differential
graded commutative associative algebra $(X, d, \bullet)$, with $d$ of
degree $+1$, and differential $\Delta$ of degree $-1$ such that $d$
graded commutes with $\Delta$ (so that $d \Delta + \Delta d =0$), and
finally the deviation $\{,\}$ of $\Delta$ from being a derivation of
$\bullet$,
\[
\{x,y\}=(-1)^{|x|}\Delta(x\bullet y)-(-1)^{|x|}\Delta(x)\bullet y -
x\bullet \Delta(y)
\]
satisfies,
\begin{eqnarray*}
  & \{x,y\}=-(-1)^{(|x|+1)(|y|+1)} \{y,x\} & \text{(Anti-symmetry),} \\
  & \{x\bullet y,z\}=x\bullet\{y,z\}+(-1)^{|y|(|z|+1)} \{x,z\}\bullet y & \text{(Leibniz relation).}
\end{eqnarray*}

The Leibniz relation can be read as saying that bracketing with a fixed
element (on the right) is a graded derivation of the product
$\bullet$. These relations imply that bracketing with a fixed element
on the left is also a graded derivation
\[
\{ x , y\bullet z \} = \{x,y\} \bullet z + (-1)^{(|x|+1)y} y \bullet
\{x,z\}
\]
and also imply that bracketing with a fixed element is a graded
derivation of the bracket,
\[
\{x,\{y,z\}\} = \{ \{x,y\} ,z\} +(-1)^{(|x|+1)(|y|+1)} \{ y,\{x,z\}\}
\quad \quad \text{(Jacobi identity).}
\]

A morphism of dBV algebras $X$ and $Y$ is a map $f: X \to Y$ that
preserves the structures $d, \Delta$, and $\bullet$.
 
\begin{rem}
  In the special case where $\mu_A$ is graded commutative, $\Delta$ becomes a
  derivation of $\bullet$ and, thus, the bracket $\{,\}$ is zero.  This
  is well known in the literature, see for example \cite{L}.  We were surprised we could not find 
 in the literature the fact that $TA$ becomes a dBV algebra when $\mu_A$ is not necessarily commutative.
  There is, however, a similar ``Lie'' version which is well known:
  the symmetric algebra of the underlying vector space of a Lie algebra is a BV algebra (see \cite{TT}). 
\end{rem}

Theorem \ref{theorem1} generalizes naturally.  If $(A,\mu_1,\mu_2,\mu_3,\ldots)$ is
an $A_\infty$ algebra, then for each $k=1,2,\ldots$, the 
linear map $\mu_k:A^{\otimes k}\to A$ can be extended to a coderivation of degree $3-2k$ of the tensor
coproduct $\Delta_{3-2k}:TA\to TA$.  In Section \ref{proof2} we show:
\begin{thm}\label{theorem2}
  If $(A,\mu_1,\mu_2,\mu_3,\ldots)$ is an $A_\infty$ algebra, 
  then $(TA,\bullet,\Delta_{1},\Delta_{-1},\Delta_{-3},\ldots)$ 
  defines a commutative $BV_\infty$ algebra. 
\end{thm}

\begin{rem}   A commutative $BV_\infty$ algebra, as defined by Kravchenko \cite{K}, is a generalization of a dBV algebra, and a special case of a
  $BV_\infty$ algebra, as shown in \cite{GTV}.  (See also \cite{DC}.)
  The precise definition is given in Section \ref{proof2}, where we show the requisite property that  $\Delta_{3-2k}$ has operator-order $k$ with respect to the shuffle product.  
  
  From a logical point of view, it is probably better to prove Theorem \ref{theorem2} first, from which 
  Theorem \ref{theorem1} follows, see Remark \ref{thm2-->thm1} below. However, we prefer to give a direct proof of Theorem \ref{theorem1} using the traditional definition of a dBV algebra, making this an easy to read self-contained section. This also has the advantage of giving an explicit formula for the bracket $\{\,,\,\}$, and gives us the opportunity to illustrate explicitly how the signs are checked in this context.
  \end{rem}

\begin{ack}
  The second author was partially supported by the Max-Planck
  Institute in Bonn, Germany.  
  The third author was supported in part by a grant from The City University of New York PSC-CUNY Research Award Program.
  We would like to thank Gabriel
  Drummond-Cole and Bruno Vallette for useful discussions about $BV_\infty$
  algebras.
\end{ack}

\section{Proof of the Theorem \ref{theorem1}}\label{proof1}

The identities $d^2=0$, $\Delta^2=0$, $\bullet$ being associative and
graded commutative, and $d$ being a derivation of $\bullet$ are all
straightforward. The (graded) anti-symmetry of the bracket follows
formally from the (graded) symmetry of $\bullet$. The functoriality
statement is immediate. It remains to show that the bracket $\{,\}$
satisfies the Leibniz relation.

We abbreviate $a_{i_1}\otimes \dots\otimes a_{i_k}$ by
$a_{i_1,\dots,i_k}$, and $\sigma^{-1}(i)$ by $\sigma^{-1}_i$ for a
permutation $\sigma\in \Sigma_k$. First, we may calculate the bracket
as
\begin{multline*}
  \{ a_{1,\dots,n},a_{n+1,\dots,n+m} \} = \sum_{\sigma\in S(n,m)} \pm
  \Delta (a_{\sigma^{-1}_1,\dots,\sigma^{-1}_{n+m}}) \\ - \left( \pm
    \Delta (a_{1,\dots,n})\bullet a_{n+1,\dots,{n+m}} \right) - \left(
    \pm a_{1,\dots,n}\bullet \Delta (a_{n+1,\dots,{n+m}}) \right)
\end{multline*}

We claim that every term in the last two expressions cancels with
precisely one term in $\sum_{\sigma\in S(n,m)} \pm \Delta
(a_{\sigma^{-1}_1,\dots,\sigma^{-1}_{n+m}})$ so that $\{
a_{1,\dots,n},a_{n+1,\dots,n+m} \} $ equals
\begin{multline*}
  \sum_{\sigma\in S(n,m)} \sum_{j\in
    C^{\{1,\dots,n\},\{n+1,\dots,n+m\}}_\sigma} \pm
  a_{\sigma^{-1}_1,\dots,\sigma^{-1}_{j-1}} \otimes
  \mu_A(a_{\sigma^{-1}_j},a_{\sigma^{-1}_{j+1}})\otimes
  a_{\sigma^{-1}_{j+1},\dots,\sigma^{-1}_{n+m}},
\end{multline*}
where the set $C^{I,J}_\sigma$ is defined, for a permutation
$\sigma\in \Sigma_k$ and disjoint set of indices $I\cup J \subseteq
\{1,\dots,k\}$ with $I\cap J=\varnothing$, by
\begin{equation*}
  C^{I,J}_\sigma =\{j\quad:\quad  \sigma^{-1}_j\in I \text{ and } \sigma^{-1}_{j+1}\in J
  \text{, or }\sigma^{-1}_j\in J \text{ and } \sigma^{-1}_{j+1}\in I\}.
\end{equation*}
In other words, $\mu_A$ is applied in the above sum whenever exactly
one of the two elements $a_{\sigma^{-1}_j}$ and
$a_{\sigma^{-1}_{j+1}}$ is taken from $a_1,\dots,a_n$, and the other
element is taken from $a_{n+1},\dots,a_{n+m}$. Since the correct terms
appear exactly once, the only difficulty is to check the cancellation
by signs, which we leave to the end of this section.

Assuming this, if we abbreviate the expression
$a_{i_1,\dots,i_{j-1}} \otimes \mu_A(a_{i_j},a_{i_{j+1}})\otimes
a_{i_{j+1},\dots,i_k}$ by $a^{(j,j+1)}_{i_1,\dots,i_k}$, then we can
write,
$$
\{ a_{1,\dots,n},a_{n+1,\dots,n+m} \}=\sum _{\sigma \in
  S(n,m)}\sum_{j\in C^{\{1,\dots,n\},\{n+1,\dots,n+m\}}_\sigma} \pm
a^{(j,j+1)}_{\sigma^{-1}_1,\dots,\sigma^{-1}_{n+m}}
$$
With this, we can check that $\{a_{1,\dots,n}\bullet
a_{n+1,\dots,n+m}, a_{n+m+1,\dots,n+m+p}\} $ equals
\begin{eqnarray*}
  &=&\sum_{\sigma\in S(n,m)} \pm \{a_{\sigma^{-1}_1,\dots,\sigma^{-1}_{n+m}},a_{n+m+1,\dots,n+m+p} \}\\
  &=&  \sum_{\rho\in S(n,m,p)} \sum_{j\in C^{\{1,\dots,n+m\},\{n+m+1,\dots,n+m+p\}}_\rho} \pm a^{(j,j+1)}_{\rho^{-1}_1,\dots,\rho^{-1}_{n+m+p}}\\
  &=& \sum_{\rho \in S(n,m,p)}\sum_{j\in C^{\{1,\dots,n\},\{n+m+1,\dots,n+m+p\}}_\rho} \pm  a^{(j,j+1)}_{\rho^{-1}_1,\dots,\rho^{-1}_{n+m+p}}\\
  &&+ \sum_{\rho \in S(n,m,p)}\sum_{j\in C^{\{n+1,\dots,n+m\},\{n+m+1,\dots,n+m+p\}}_\rho} \pm a^{(j,j+1)}_{\rho^{-1}_1,\dots,\rho^{-1}_{n+m+p}}\\
  &=& a_{1,\dots,n}\bullet\{a_{n+1,\dots,n+m},a_{n+m+1,\dots,n+m+p}\}\\ 
  && \hspace{1.5in}  \pm  \{a_{1,\dots,n},a_{n+m+1,\dots,n+m+p}\}\bullet a_{n+1,\dots,n+m},
\end{eqnarray*}
where $S(n,m,p) \subseteq \Sigma_{n+m+p}$ consists of those permutations
$\rho\in \Sigma_{n+m+p}$ that satisfy $ \rho(1)<\dots <\rho(n),
\rho(n+1)<\dots <\rho(n+m)$, and $\rho(n+m+1)<\dots <\rho(n+m+p)$. By
a careful consideration of the signs similar to the check below, it follows that the
Leibniz identity holds.

Now, we check the sign mentioned above. If we shuffle $a_{n+1, \ldots ,n+j}$ past $a_i$, for $1 \leq i \leq n$
and $1 \leq j \leq m$, and then apply $\Delta$, we obtain the term
\[
a_{1, \ldots, i-1} \ot a_{n+1, \ldots , n+j} \ot \mu_A(a_i , a_{i+1})
\ot a_{i+2, \ldots n} \ot a_{n+j+1, \ldots , n+m}
\]
with sign
\[
(-1)^{(|a_i| + \dots + |a_n| + n-i + 1)( |a_{n+1}| + \dots + |a_{n+j}|
  + j) + ( |a_1| + \dots + |a_{i-1} | + | a_{n+1}| + \dots + |a_{n+j}|
  + |a_i| + (i-1+j))}
\]
while in the other order, $\Delta$ then shuffle, we obtain the same
term with sign
\[
(-1)^{ ( |a_1| + \dots + |a_{i}| + i + 1) + (|\mu(a_i,a_{i+1}) | +
  |a_{i+2}| + \dots + |a_n| + n - i)( |a_{n+1}| + \dots + |a_{n+j}| +
  j) }
\]
and these agree. This special case implies the general case, for any
shuffle, since a more general shuffle introduces the same additional
sign in both cases.

Similarly, shuffling $a_{i+1, \ldots,n}$ past $a_{n+j+1}$ for $1 \leq
i < n$ and $1 \leq j < m$, and then applying $\Delta$, we obtain the
term
\[
a_{1, \ldots, i} \ot a_{n+1, \ldots ,n+j-1} \ot \mu_A(a_{n+j},
a_{n+j+1}) \ot a_{i+1, \ldots, n} \ot a_{n+j+2, \ldots, n+m}
\]
with sign
\[
(-1)^{(|a_{i+1}| + \dots + |a_n| + n-i)(|a_{n+1}| + \dots +
  |a_{n+j+1}| + j+1) +(|a_1| + \dots + |a_i| + |a_{n+1}| + \dots +
  |a_{n+j}| + i+j + 1) }
\]
while in the other order we obtain the same term with sign
\[
(-1)^{( |a_{n+1}| + \dots + |a_{n+j-1}| + j-1) + (|a_{i+1}| + \dots +
  |a_n| + n-i)(|a_{n+1}| + \dots + |a_{n+j-1}| +
  |\mu(a_{n+j},a_{n+j+1})| + j ) }
\]
These differ by $(-1)^{|a_1| + \dots + |a_n| +n}$, as expected. Again,
this special case implies the general case, as before. This completes the proof of Theorem \ref{theorem1}.

\section{Proof of Theorem \ref{theorem2}}\label{proof2}
Let $(X,\bullet)$ be a graded commutative associative algebra.  An operator
$\Delta:X\to X$ has operator-order $n$ if and only if
\begin{equation*}
\sum
(-1)^{n+1-r+\kappa} \Delta\left (x_{i_1}\bullet \cdots \bullet
  x_{i_r}\right) \bullet x_{i_{r+1}}\bullet \cdots \bullet
x_{i_{n+1}}=0
\end{equation*}
where the sum is taken over nonempty subsets $\{i_1,\ldots,
i_r:i_1<\ldots<i_r\}\subseteq \{1,\ldots, n+1\}$ and $\{1,\ldots,
n+1\}\setminus \{i_1, \ldots, i_r\}$ has been ordered 
$i_{r+1}<\cdots <i_{n+1}$, and $\kappa$ comes from the usual Koszul sign rule.

If $\Delta$ has operator-order one, then it is a
derivation of $\bullet$.  If $\Delta$ has operator-order two, then
its deviation from being a derivation of $\bullet$, 
is a derivation of $\bullet$.  This means that if we define
$\{\,,\,\}$ to be the deviation of $\Delta$ from being a derivation of
$\bullet$, then $\{\,,\,\}$ and $\bullet$ satisfy the Leibniz
relation.

\begin{rem}\label{thm2-->thm1}
Using this fact, one can prove Theorem \ref{theorem1} without reference to
the bracket---here is an outline: any map $\mu_A: A\otimes A \to A$
 becomes an order $2$ operator $\Delta:TA\to TA$ with respect to the
 shuffle product when it is lifted as a coderivation of the tensor
 coproduct (as we will show in the lemma below). It is straightforward to check that $\mu_A$ being associative implies that $\Delta^2=0$, since $\Delta^2$ is the lift of the associator of $\mu_A$ to a coderivation. So, if $(A,d_A,\mu_A)$ is a differential graded algebra, with $\mu_A$ of degree zero, $\Delta$ has degree $-1$, and since $d_A$ is a derivation of $\mu_A$, then $d:TA\to TA$ and $\Delta:TA \to TA$ commute.  That proves that $(TA,d,\Delta,\bullet)$ is a dBV algebra. 
\end{rem}
 
To generalize: a Kravchenko commutative $BV_\infty$ algebra consists of
 a graded commutative
differential graded algebra $(X,d,\bullet)$ and a collection
$\{\Delta_k:X\to X\}_{k=1,-1,-3,-5,\ldots}$ of
operators satisfying
\begin{itemize}
\item $\Delta_{1}=d$,
\item each $\Delta_{3-2k}$ has degree $3-2k$ and operator-order $k$,
\item for each $n$, $\sum_{j+k=n}\Delta_{j}\Delta_{k}=0$.
\end{itemize}
We use
the degree convention in \cite{K} but note that in \cite{GTV} the
opposite convention is used (there, $d$ has degree $-1$ and the higher $\Delta$ operators have positive degree).  As a special case, a dBV algebra is a 
Kravchenko commutative $BV_\infty$ algebra with
$\Delta_{-3}=\Delta_{-5}=\cdots=0.$

To prove Theorem \ref{theorem2}, assume that $(A,\mu_1,\mu_2, \mu_3 \ldots)$
is an $A_\infty$ algebra.  By definition of an $A_\infty$ algebra, each $\mu_k$ lifts to a
degree $3-2k$ coderivation $\Delta_{3-2k}:TA\to TA$, with
$\Delta_{1}=d$ and relations $\sum_{j+k=n}\Delta_{j}\Delta_{k}=0$.
Thus it only remains to prove, that each $\Delta_{3-2k}$ has order $k$ with respect to the shuffle product $\bullet$.  This follows from the following general lemma.
\begin{lemma}
Let $f:A^{\ot n}\to A$ be any linear map and let $F:TA\to TA$
  be the lift of $f$ to a coderivation.  Then $F$ has
  order $n$ with respect to the shuffle product.
\end{lemma}
\begin{proof}
Let $X^1, \ldots, X^{n+1}$ be monomials in $TA.$  
So, $X^i=a^i_{1}\otimes \cdots \otimes a^i_{s_i}$ with each $a^i_\ell \in
A$.  Then $(-1)^{n+1-r+\kappa} F(X^{i_1} \bullet \cdots \bullet X^{i_r})\bullet X^{i_{r+1}}\dots\bullet X^{i_{n+1}}$ consists of a sum of terms of
the form
\begin{equation}\label{EQ1}
\pm \dots\ot f(a^{i'_1}_{\ell_1} \ot \cdots \ot a^{i'_n}_{\ell_n})\ot \dots
 \text{ (the rest of the $a^i_\ell$'s are outside of $f$)},
\end{equation} 
where $f$ is applied to $a^{i'_1}_{\ell_1} \otimes \cdots \ot a^{i'_n}_{\ell_n}$, and the remaining tensor products are applied outside of $f$. The list $\{i'_1, \ldots,
i'_n\}$ may contain repetition, and we may order the list from smallest to largest
without repetition as $\{i_1, \ldots, i_k \}$.  Every term of the form \eqref{EQ1} which contains only the indices $\{i_1, \ldots, i_k \}$ inside $f$, appears for each index set $J=\{j_1,\dots,j_q\}$ with $\{i_1,\dots, i_k\}\subseteq J \subseteq \{1,\dots, n+1\}$ exactly once in the sum of $(-1)^{n+1-q+\kappa} F(X^{j_1} \bullet \cdots \bullet X^{j_q})\bullet X^{j_{q+1}}\dots\bullet X^{j_{n+1}}$. Now, for a fixed expression in Equation \eqref{EQ1} induced by different index sets $J$, the only difference in the sign of \eqref{EQ1} is a factor of $(-1)^q$, where $q=|J|$, and all other signs coincide for varying $J$. We thus need to show that summing $(-1)^{|J|}$ over all $J$ with $\{i_1,\dots, i_k\}\subseteq J \subseteq \{1,\dots, n+1\}$ vanishes. Since there are exactly $n+1-k$ choose $q-k$ such subsets $J$ with $q$ elements, we obtain that
\begin{multline*}
\sum_J (-1)^{|J|}=\sum_{q=k}^{n+1} \left(\begin{matrix} n+1-k \\ q-k \end{matrix}\right)\cdot (-1)^q =(-1)^k \cdot \sum_{q'=0}^{n+1-k} \left(\begin{matrix} n+1-k \\ q' \end{matrix}\right)\cdot  (-1)^{q'}
\\
=(-1)^k \cdot (-1+1)^{n+1-k} =0,
\end{multline*}
where we used the binomial theorem in the second to last equality. This completes the proof of the lemma.
\end{proof}

\bibliographystyle{plain}
\bibliography{TTW.bib}

\begin{thebibliography}{1}

\bibitem{DC}
Gabriel Drummond-Cole.
\newblock {\em Homotopy {B}atalin-{V}ilkovisky algebras, trivializing circle
  actions, and moduli space}.
\newblock PhD thesis, The City University of New York, 2010.

\bibitem{GTV}
Imma {Galvez-Carrillo}, Andy {Tonks}, and Bruno {Vallette}.
\newblock {Homotopy Batalin-Vilkovisky algebras}.
\newblock {\em ArXiv e-prints}, July 2009.

\bibitem{K}
Olga Kravchencko.
\newblock Deformations of {B}atalin-{V}ilkovisky algebras.
\newblock In {\em Poisson Geometry}, Banach Center Publ., pages 131--139.
  Polish Acad. Sci., 2000.

\bibitem{L}
Jean-Louis Loday.
\newblock {\em Cyclic Homology}.
\newblock Number 301 in Grundlehren der mathematischen Wissenschaften. Springer
  Verlag, 1992.

\bibitem{M}
Saunders MacLane.
\newblock {\em Homology}.
\newblock Springer Verlag, 1995.
\newblock Reprint of 1975 edition.

\bibitem{S}
James~D. Stasheff.
\newblock The intrinsic bracket on the deformation complex of an associative
  algebra.
\newblock {\em J. Pure Appl. Algebra.}, 89(1-2):231--235, 1993.

\bibitem{TT}
Dmitry Tamarkin and Boris Tsygan.
\newblock Noncommutative differential calculus, homotopy {BV} algebras and
  formality conjectures.
\newblock {\em Methods Func. Anal. Topology}, 6(2):85--100, 2000.

\end{thebibliography}

\end{document}